\documentclass[12pt,reqno,a4paper]{amsart}
\usepackage{amscd,amsfonts,amssymb,amsmath,amsthm}
\usepackage[arrow,matrix]{xy}
\usepackage{graphicx,epstopdf}
\usepackage{color}
\usepackage{cite}
\usepackage{mathrsfs,euscript,hyperref}

\newtheorem{theorem}{Theorem}[section]
\newtheorem{lemma}[theorem]{Lemma}

\newtheorem{cor}[theorem]{Corollary}
\theoremstyle{remark}
\newtheorem{remark}{Remark}[section]

\theoremstyle{definition}

\numberwithin{equation}{section} 

\newcommand{\bea}{\begin{eqnarray}}
\newcommand{\eea}{\end{eqnarray}}

\usepackage{cancel}
\usepackage[normalem]{ulem}

\begin{document}
\title[The global attractiveness of the fixed point]
{The global attractiveness of the fixed point of a gonosomal evolution operator}

\author{Akmal T.~Absalamov}
\address{Samarkand State University,
Boulevard str., 140104,
Samarkand, Uzbekistan.}
\email{absalamov@gmail.com}

\date{}

\maketitle{}
\begin{abstract}
In the paper we prove a conjecture of U.A. Rozikov and R. Varro
about globally attractiveness of a unique fixed point of the
normalized evolution operator of a sex linked inheritance.
\end{abstract}
{\bf Keywords:}
Bisexual population, gonosomal operator, fixed point, trajectory.

\section{Introduction}\label{int}

\noindent
A {\it population} is a set of organisms of the same kind, some long time living in one territory
(occupying a particular area) and are completely isolated from other the same groups.

In the life sciences the {\it population dynamics} branch studies the
size and age composition of populations as dynamical systems. These investigations are
motivated by their application to population growth,  ageing populations,  or population decline.

The population dynamics is a well developed branch of mathematical biology, which has a
history of more than two hundred years \cite{Ba}, although more recently
the branch of mathematical biology has greatly increased.
Many concrete models of mathematical biology
described by corresponding non-linear evolution operator.
 Since there is no any general theory of non-linear operators, for each concrete such operator
 one has to use a specific method of investigation.

 In this paper we also study dynamical system
 generated by a concrete non-linear multidimensional operator describing a gonosomal evolution.
Our model is related to a bisexual population. We note that investigation of dynamical systems
generated by evolution operators of free and bisexual population can be reduced to the study of
nonlinear dynamical systems (see \cite{Ganikhodzhaev.book.2011}, \cite{K}, \cite{Lyubich.book.1992},
\cite{Rozikov.book.2013} for more details).

In biology sex is determined genetically, males and females have
different alleles or even different genes that specify their
sexual morphology. In animals this is often accompanied by
chromosomal differences. Determination genetically is generally
through chromosome combinations of $XY$ (for example: humans,
mammals), $ZW$ (birds). Generally in this method, the sex is
determined by amount of genes expressed across the two chromosomes.
There are some sex linked systems which depend on temperature and
even some of systems have sex change phenomenon, see
\cite{Rozikov.book.2013} for more details. In \cite{Varro.book.2015}
an algebra associated to a sex change is constructed.

In bisexual population any kind of differentiation must agree with
the sex differentiation, i.e. all the organisms of one type must
belong to the same sex. Thus it is possible to speak of male and
female types. For mathematical models of bisexual population, see
\cite{Ladra.book.2013}, \cite{Lyubich.book.1992}, \cite{Reed.book.1997},
\cite{Rozikov.book.2011}.

Sex is controlled by two chromosomes called gonosomes. Gonosomal
inheritance is a mode of inheritance that is observed for traits
related to a gene encoded on the sex chromosomes.

We discuss one example of sex-linked inheritance. Hemophilia
is a lethal recessive $X$-linked disorder: a female carrying two
alleles for hemophilia die. Therefore, if we denote by $X^h$ the
gonosome  $X$ carrying the hemophilia, there are only two female
genotypes: $XX$ and $XX^h$ ($X^hX^h$ is lethal) and two male
genotypes: $XY$ and $X^hY$. We have four types of crosses defined
as
\begin{align*}
XX   \times{XY}   &\rightarrow \frac{1}{2}XX,\: \frac{1}{2}XY, \\
XX   \times{X^hY} &\rightarrow \frac{1}{2}XX^h,\: \frac{1}{2}XY, \\
XX^h \times{XY}   &\rightarrow \frac{1}{4}XX,\: \frac{1}{4}XX^h,\:
\frac{1}{4}XY, \frac{1}{4}X^hY, \\
XX^h \times{X^hY} &\rightarrow \frac{1}{3}XX^h, \: \frac{1}{3}XY, \:
\frac{1}{3}X^hY.
\end{align*}

Let $F=\{XX,XX^h\}$ and $M=\{XY,X^h Y\}$ be sets of genotypes.
Assume that state of the set $F$ is given by a real vector $(x,y)$
and state of $M$ by a real vector $(u,v)$. Then a state of the set
$F\cup{M}$ is given by the vector $t=(x,y,u,v)\in{\mathbb{R}^4}$.

We consider the following subsets of $\mathbb{R}^4$:
$$S^{3}=\Bigl\{s=(x,y,u,v)\in{\mathbb{R}^{4}}: \, x\geq0, \, y\geq0,
\, u\geq0, \, v\geq0,  \, x+y+u+v=1 \Bigl\}$$ the three-dimensional simplex;
$$\Theta=\{s=(x,y,u,v)\in{S^{3}}: (x,y)=(0,0) \,\, or\,\, (u,v)=(0,0)\};$$
$$S^{2,2}=S^{3}\setminus\Theta.$$

If $t'=(x',y',u',v')$ is a state of the system $F\cup{M}$ in the
next generation, then by the above rule we get the evolution
operator $W:S^{2,2}\rightarrow{S^{2,2}}$ defined by
\begin{equation}\label{A1}
W: \left\{
\begin{array}{ll}
\begin{aligned}
& x' &= \quad &\frac{2xu+yu}{4(x+y)(u+v)}, \\[2mm]
& y' &= \quad &\frac{6xv+3yu+4yv}{12(x+y)(u+v)}, \\[2mm]
& u' &= \quad &\frac{6xu+6xv+3yu+4yv}{12(x+y)(u+v)}, \\[2mm]
& v' &= \quad &\frac{3yu+4yv}{12(x+y)(u+v)}.
\end{aligned}
\end{array}
 \right.
\end{equation}
\begin{remark} For a general sex-linked population the non-linear evolution operator
 first derived by Kesten \cite{K}. The operator \eqref{A1}
obtained from the operator of Kesten by choosing
appropriate coefficients for the hemophilia.
\end{remark}

The main problem for a given operator $W$ and arbitrarily initial point
$s^{(0)}\in S^{2,2}$, is to describe the limit points of the trajectory
$\{s^{(m)}\}_{m=0}^{\infty}$, where $s^{(m)}=W^{m}(s^{(0)})=\underbrace {W(W(...W}_{m}(s^{(0)}))...)$.

\begin{remark}
In \cite{AR} unnormalized form of the operator \eqref{A1}
on $\mathbb{R}^{4}$ is considered with arbitrarily coefficients  and obtained
various conditions which lead to the set of limit points  being the
origin or infinity.
\end{remark}

\begin{remark}
In their work \cite{Rozikov.book.2016}
U.A. Rozikov and R. Varro considered normalized gonosomal evolution
operator \eqref{A1} of a sex linked inheritance. Mainly they studied
dynamical systems of a hemophilia which is biological group of disorders
connected with genes that diminish the body's ability to control blood
clotting or coagulation that is used to stop bleeding when a blood
vessel is broken. They proved that the operator $W$ has a unique
nonhyperbolic fixed point $s_0=(\frac{1}{2},0,\frac{1}{2},0)$ and there
is an open neighborhood $\cup(s_0)\subset{S^{2,2}}$ of $s_0$ such that
for any initial point $s\in{\cup(s_0)}$, the limit point of trajectories
$\{W^{m}(s)\}$ tends to $s_0$. Moreover they made a conjecture for an
initial point $s\in{S^{2,2}}$. In this article we give a proof of that
conjecture.
\end{remark}

\section{Results}

The main achievement of the present manuscript is the following
result which is given as a conjecture by U.A. Rozikov and R.
Varro in \cite{Rozikov.book.2016}.

\begin{theorem}\label{thm:main}
The operator $W:S^{2,2}\rightarrow{S^{2,2}}$ given by \eqref{A1}
has unique nonhyperbolic fixed point $s_0=(\frac{1}{2},0,
\frac{1}{2},0)$ and for any initial point $s\in{S^{2,2}}$ we
have
\begin{equation}\label{A6}
\lim\limits_{m\rightarrow \infty}{W^{m}(s)}=s_0=(\frac{1}
{2},0,\frac{1}{2},0).
\end{equation}
\end{theorem}

Let $s^{(0)}=(x^{(0)},y^{(0)},u^{(0)},v^{(0)})\in{S^{2,2}}$ be an initial
state (the probability distribution on the set $\{XX,XX^{h};XY,X^{h}Y\}$
of genotypes). This Theorem~{\upshape{\ref{thm:main}}} has
the following biological interpretations: when time goes to infinity, the population tends to the equilibrium state
$s_0=(\frac{1}{2},0,\frac{1}{2},0)$, meaning that the future of the population is stable: genotypes $XX$ and
$XY$ are survived always, but the genotypes $XX^{h}$ and $X^{h}Y$
 will asymptotically disappear. Consequently, only healthy chromosomes will survive.

From biological interpretations of the result \eqref{A6} we can see that
problem of investigating the behavior of trajectories of the
operator \eqref{A1} is great importance in understanding of the
hemophilia at a sex linked inheritance.

Throughout this section
for the trajectories we use the notation
\begin{equation*}
s^{(m)}=(x^{(m)},y^{(m)},
u^{(m)},v^{(m)})=W^{m}(s), \quad m=0,1,2,...
\end{equation*}
We present several lemmas which give useful
estimates and help to prove the Theorem~{\upshape{\ref{thm:main}}}.
\begin{lemma}\label{lem:maj}
Let $s^{(0)}=(x^{(0)}, y^{(0)}, u^{(0)}, v^{(0)})\in{S^{2,2}}$ be any
initial point. Then, for all non-negative integers $m$, it holds that

(i) $x^{(m+1)}\leq{u^{(m+1)}}$ and $v^{(m+1)}\leq{y^{(m+1)}}\leq{u^{(m+1)}}$

(ii)  and that
	\begin{equation*}
	\begin{aligned}
	&\frac{1}{8} \leq \frac{u^{(m+1)}}{4(u^{(m+1)}+v^{(m+1)})} \leq x^{(m+2)}
\leq \frac{u^{(m+1)}}{2(u^{(m+1)}+v^{(m+1)})} \leq \frac{1}{2},\\
	&0\; \leq \frac{v^{(m+1)}}{3(u^{(m+1)}+v^{(m+1)})} \leq y^{(m+2)}
\leq \frac{u^{(m+1)}+2v^{(m+1)}}{4(u^{(m+1)}+v^{(m+1)})}  \leq \frac{1}{2},\\
	&\frac{1}{4} \leq \frac{2x^{(m+1)}+y^{(m+1)}}{4(x^{(m+1)}+y^{(m+1)})} \leq u^{(m+2)}
\leq \frac{3x^{(m+1)}+2y^{(m+1)}}{6(x^{(m+1)}+y^{(m+1)})} \leq \frac{1}{2},\\
	&0\; \leq \frac{y^{(m+1)}}{4(x^{(m+1)}+y^{(m+1)})} \leq v^{(m+2)}
\leq \frac{y^{(m+1)}}{3(x^{(m+1)}+y^{(m+1)})}   \leq\frac{1}{3}.
	\end{aligned}
	\end{equation*}
\end{lemma}
\begin{proof}
We provide a sketch of the proof for the first part only. The claim
in the second part immediately follows from the first part. In view
of \eqref{A1}, we have

\begin{equation}\label{A40}
\begin{cases}
\displaystyle
x^{(m+1)}=\frac{2x^{(m)}u^{(m)}+y^{(m)}u^{(m)}}{4(x^{(m)}+y^{(m)})(u^{(m)}+v^{(m)})}, \\[3.5mm]
\displaystyle
y^{(m+1)}=\frac{6x^{(m)}v^{(m)}+3y^{(m)}u^{(m)}+4y^{(m)}v^{(m)}}{12(x^{(m)}+y^{(m)})(u^{(m)}+v^{(m)})}, \\[3.5mm]
\displaystyle
u^{(m+1)}=\frac{6x^{(m)}u^{(m)}+6x^{(m)}v^{(m)}+3y^{(m)}u^{(m)}+4y^{(m)}v^{(m)}}{12(x^{(m)}+y^{(m)})(u^{(m)}+v^{(m)})}, \\[3.5mm]
\displaystyle
v^{(m+1)}=\frac{3y^{(m)}u^{(m)}+4y^{(m)}v^{(m)}}{12(x^{(m)}+y^{(m)})(u^{(m)}+v^{(m)})}.
\end{cases}
\end{equation}

It is thus not difficult to see that
\begin{equation*}
\begin{aligned}
u^{(m+1)}-x^{(m+1)}&=\frac{v^{(m)}(3x^{(m)}+2y^{(m)})}{6(x^{(m)}+y^{(m)})(u^{(m)}+v^{(m)})} \geq 0,\\
u^{(m+1)}-y^{(m+1)}&=\frac{x^{(m)}u^{(m)}}{2(x^{(m)}+y^{(m)})(u^{(m)}+v^{(m)})} \geq 0,\\
y^{(m+1)}-v^{(m+1)}&=\frac{x^{(m)}v^{(m)}}{2(x^{(m)}+y^{(m)})(u^{(m)}+v^{(m)})} \geq 0,
\end{aligned}
\end{equation*}
which complete the proof of the first part.
\end{proof}
Now we make the notations
\begin{equation}\label{not.al.be}
\alpha^{(m)}:=\frac{y^{(m+2)}}{x^{(m+2)}}, \quad
\beta^{(m)}:=\frac{v^{(m+2)}}{u^{(m+2)}}, \quad m=0, 1,\ldots
\end{equation}
\begin{lemma} For any initial point $s^{(0)}=(x^{(0)}, y^{(0)}, u^{(0)}, v^{(0)})\in{S^{2,2}}$
and for any nonnegative integer $m$ the following hold:
\begin{equation}\label{al.be.bdd.seq}
0\leq\alpha^{(m)}\leq4, \quad 0\leq\beta^{(m)}\leq1
\end{equation}
\end{lemma}
\begin{proof}
From the first and second inequalities of the part {\it (ii)} of the Lemma~\ref{lem:maj},
for any initial point $s^{(0)}=(x^{(0)}, y^{(0)}, u^{(0)}, v^{(0)})\in{S^{2,2}}$
and for any nonnegative integer $m$ we obtain
\begin{equation*}
0=\frac{min{\bigl\{y^{(m+2)}\bigl\}}}{max{\bigl\{x^{(m+2)}
\bigl\}}}\leq\alpha^{(m)}\leq\frac{max{\bigl\{y^{(m+2)}\bigl\}}}{min{\bigl\{x^{(m+2)}
\bigl\}}}=4.
\end{equation*}
Since for any initial point $s^{(0)}=(x^{(0)}, y^{(0)}, u^{(0)}, v^{(0)})\in{S^{2,2}}$
and for any nonnegative integer $m$, the part {\it (i)} of the Lemma~\ref{lem:maj} gives us the
inequality $v^{(m+2)}\leq{u^{(m+2)}}$ then $$\beta^{(m)}=\frac{v^{(m+2)}}{{u^{(m+2)}}}\leq1.$$
Moreover, from the third and fourth inequalities of the
part {\it (ii)} of the Lemma~\ref{lem:maj} we get the lower bound for $\beta^{(m)}$:
$$\beta^{(m)}\geq\frac{min{\bigl\{v^{(m+2)}\bigl\}}}{max{\bigl\{u^{(m+2)}\bigl\}}}=0.$$
This completes the proof.
\end{proof}
Next, using the system of equations \eqref{A40}, we obtain

\begin{equation}\label{alph.beta.rec}
\begin{aligned}
\alpha^{(m+1)}&=\frac{6\beta^{(m)}+3\alpha^{(m)}+4\alpha^{(m)}\beta^{(m)}}{6+3\alpha^{(m)}},\\
\beta^{(m+1)}&=\frac{3\alpha^{(m)}+4\alpha^{(m)}\beta^{(m)}}{6+6\beta^{(m)}+3\alpha^{(m)}+4\alpha^{(m)}\beta^{(m)}}
\end{aligned}
\end{equation}
which yields the nonlinear dynamical system
\begin{equation}\label{A7}
F:
\begin{cases}
\displaystyle
\alpha'=\frac{6\beta+3\alpha+4\alpha\beta}{6+3\alpha}, \\[2ex]
\displaystyle
\beta'=\frac{3\alpha+4\alpha\beta}{6+6\beta+3\alpha+4\alpha\beta}.
\end{cases}
\end{equation}
with the initial point $(\alpha^{(0)}, \beta^{(0)})\in\Delta$, where
\begin{equation}
\Delta:=\{(\alpha,\beta)\in{\mathbb{R}^{2}}: 0\leq\alpha\leq{4},\,\, 0\leq\beta\leq{1}\}.
\end{equation}
It is not hard to see that $(0,0)$ is the unique non-hyperbolic fixed
point of $F$ with the eigenvalues $\lambda_1=1$, $\lambda_2=-\frac{1}{2}$.
Here we recall that a fixed point of the operator $F$ is called hyperbolic if its
Jacobian at the fixed point has no eigenvalues on the unit circle.
\begin{remark}
Fixed point $(\alpha,\beta)=(0,0)$ of the dynamical system \eqref{A7}
corresponds to the fixed point $s_0=(\frac{1}{2},0,\frac{1}{2},0)$ of
the dynamical system \eqref{A1}. For the uniqueness of the fixed point of
the operator $W$, see \cite{Rozikov.book.2016}.
\end{remark}

\begin{lemma}\label{AA}
For any initial point $(\alpha,\beta)\in\Delta$, we have

(i) $F(\alpha,\beta)\in\Omega\subset{\Delta}$, where
$$\Omega=\{(\alpha,\beta)\in
{\mathbb{R}^{2}}: 0\leq\alpha\leq{2},\,\, 0\leq\beta\leq{1}\};$$
(ii) $\beta'\leq\alpha'$ and $\alpha^{(2)}+\beta^{(2)}\leq{\alpha'+\beta'},$
where $\alpha', \beta'$ are defined by \eqref{A7}.
\end{lemma}
\begin{proof}
For the proof of the first part it suffices to observe that
\begin{equation*}
0\leq\alpha'=\frac{6\beta+3\alpha+4\alpha\beta}{6+3\alpha}
=2-\frac{6(1-\beta)+4\alpha(1-\beta)+(6-\alpha)}{6+3\alpha}\leq2
\end{equation*}
and that
\begin{equation*}
0\leq\beta'=\frac{3\alpha+4\alpha\beta}{6+6\beta+3\alpha+4\alpha\beta}
=1-\frac{6+6\beta}{6+6\beta+3\alpha+4\alpha\beta}\leq1,
\end{equation*}
while the first claim of the second part follows by observing that
\begin{equation*}
\begin{aligned}
\alpha'-\beta'&=\frac{6\beta+3\alpha+4\alpha\beta}{6+3\alpha}-\frac{3\alpha+4\alpha\beta}{6+6\beta+3\alpha+4\alpha\beta}\\
&=\frac{4(9\beta+9\beta^{2}+9\alpha\beta+12\alpha\beta^{2}+	3\alpha^{2}\beta+4\alpha^{2}\beta^{2})}{3(2+\alpha)(6+6\beta+
	3\alpha+4\alpha\beta)}\geq0.
\end{aligned}
\end{equation*}
The last claim follows from the relations
\begin{equation*}
\begin{aligned}
\alpha'+\beta'-(\alpha^{(2)}+\beta^{(2)})&=\alpha'+\beta'-\Bigl(\frac{6\beta'+3\alpha'+4\alpha'\beta'}{6+3\alpha'}+\frac{3\alpha'+	 4\alpha'\beta'}{6+6\beta'+3\alpha'+4\alpha'\beta'}\Bigl)\\
&\geq\frac{12(\alpha'-\beta')+6(\alpha'^{2}-\beta'^{2})+4\alpha'\beta'
	(\alpha'-\beta')
}{3(2+\alpha')(6+6\beta'+3\alpha'+4\alpha'\beta')}
\end{aligned}
\end{equation*}
and the fact that $\alpha'\geq\beta'$.
\end{proof}

\begin{cor}\label{cor:crucial}
For any initial point $(\alpha,\beta)\in\Delta$, it holds that
\begin{equation}\label{ineq.be.alph.pow}
0\leq\beta^{(m)}\leq\alpha^{(m)}, \quad m=1, 2, \ldots
\end{equation}
and that
\begin{equation}
\alpha^{(m+1)}+\beta^{(m+1)}\leq{\alpha^{(m)}+\beta^{(m)}},  \quad m=1, 2, \ldots
\end{equation}
In particular, the sequence $\{\alpha^{(m)}+\beta^{(m)}\}_{m\geq1}$ is convergent.
\end{cor}

\begin{lemma}\label{BB}
For any initial point $(\alpha,\beta)\in\Delta$
the following hold
\begin{equation}\label{A10}
\lim\limits_{m\rightarrow \infty}\alpha^{(m)}=\lim\limits_
{m\rightarrow \infty}\beta^{(m)}=0.
\end{equation}
\end{lemma}
\begin{proof}
We recall the boundedness of these sequences from \eqref{al.be.bdd.seq}. Hence Bolzano--Weierstrass theorem ensures the
existence of real numbers $a\in[0,4]$, $b\in[0,1]$ and subsequences
$\{\alpha^{(m_k)}\}_{k\geq1}$ and $\{\beta^{(m_k)}\}_{k\geq1}$ such that
\begin{equation}\label{lims.subseq}
\lim_{k\to\infty}\alpha^{(m_k)}=a, \quad
\lim_{k\to\infty}\beta^{(m_k)}=b.
\end{equation}
Since the sequence $\{\alpha^{(m)}+\beta^{(m)}\}_{m\geq1}$ is convergent
(see Corollary~\ref{cor:crucial}), we deduce from \eqref{lims.subseq} that

\begin{equation}\label{lim.of.sum}
\begin{aligned}
\lim\limits_{m\rightarrow \infty}\Bigl
(\alpha^{(m)}+\beta^{(m)}\Bigl)&=\lim\limits_{k\rightarrow
\infty}\Bigl(\alpha^{(m_{k}+1)}+\beta^{(m_{k}+1)}\Bigl)=
\lim\limits_{k\rightarrow \infty}\Bigl(\alpha^{(m_k)}+
\beta^{(m_k)}\Bigl)=\\&=\lim\limits_{k\rightarrow \infty}
\alpha^{(m_k)}+\lim\limits_{k\rightarrow \infty}\beta^{(m_k)}=a+b.
\end{aligned}
\end{equation}
Furthermore, \eqref{ineq.be.alph.pow} and \eqref{lims.subseq} imply that
\begin{equation}\label{b.leq.a}
0\leq{b}\leq{a}.
\end{equation}	
Next, in view of \eqref{alph.beta.rec}, we can write
\begin{align*}
\alpha^{(m_{k}+1)}+\beta^{(m_{k}+1)}&=\frac{6\beta^{(m_k)}+3\alpha^{(m_k)}
+4\alpha^{(m_k)}\beta^{(m_k)}}{6+3\alpha^{(m_k)}}+ \\
&+\frac{3\alpha^{(m_k)}+4\alpha^{(m_k)}\beta^{(m_k)}}{6+6\beta^{(m_k)}+
3\alpha^{(m_k)}+4\alpha^{(m_k)}\beta^{(m_k)}}.
\end{align*}
Letting $k\to\infty$ and using \eqref{lims.subseq}, \eqref{lim.of.sum}, we
get the equation
\begin{equation}\label{A8}
a+b=\frac{6b+3a+4ab}{6+3a}+\frac{3a+4ab}{6+6b+3a+4ab}
\end{equation}
which can be written, equivalently, as
\begin{equation*}
a\bigl(12(a-b)+6(a^{2}-b^{2})+4ab(a-b)+6a+3a^{2}+15ab+8a^{2}b\bigr)=0.
\end{equation*}	
On the account of the constraint \eqref{b.leq.a}, it follows that the latter
equation has a unique solution $$(a,b)=(0,0).$$ In particular, we have
\begin{equation}
\lim_{m\to\infty}\bigl(\alpha^{(m)}+\beta^{(m)}\bigl)=0.
\end{equation}
However, this observation together with the inequality
\begin{equation*}
0\leq\beta^{(m)}\leq\alpha^{(m)}\leq \alpha^{(m)}+\beta^{(m)},
\quad m=1, 2, \ldots
\end{equation*}
imply that both of the sequence $\{\alpha^{(m)}\}_{m\geq1}$ and $\{\beta^{(m)}\}_{m\geq1}$ converge to zero as $m\to\infty$.\\
This completes the proof.
\end{proof}
\begin{cor} The result \eqref{A10} gives
\begin{equation}\label{conc.1}
\lim\limits_
{m\rightarrow \infty}y^{(m)}=\lim\limits_{m\rightarrow \infty}v^{(m)}=0.
\end{equation}
\end{cor}
Indeed, from Lemma~\ref{lem:maj}(ii) and \eqref{not.al.be} we write
$$0\leq y^{(m+2)}=\alpha^{(m)}x^{(m+2)}\leq\frac{1}{2}\alpha^{(m)}, \quad m=0,1, 2, \ldots$$
$$0\leq v^{(m+2)}=\beta^{(m)}u^{(m+2)}\leq\frac{1}{2}\beta^{(m)}, \quad m=0,1, 2, \ldots$$
Hence, \eqref{conc.1} holds.\\
On the other hand, in view of \eqref{A40}, we have
\begin{equation*}
x^{(m+3)}=\frac{2+\alpha^{(m)}}{4(1+\alpha^{(m)})(1+\beta^{(m)})},
\quad
u^{(m+3)}= \frac{6+6\beta^{(m)}+3\alpha^{(m)}+4\alpha^{(m)}\beta^{(m)}}{12(1+\alpha^{(m)})(1+\beta^{(m)})},
\end{equation*}
implying the convergence of the sequences $\{x^{(m)}\}_{m\geq1}$ and
$\{u^{(m)}\}_{m\geq1}$ with
\begin{equation}\label{conc.2}
\lim_{m\to\infty}x^{(m)}=\lim_{m\to\infty}u^{(m)}=\frac{1}{2}.
\end{equation}
Now \eqref{A6} follows from \eqref{conc.1} and \eqref{conc.2}.


\end{document}